\documentclass[a4paper, 12pt, twoside]{article}
\usepackage{latexsym,bm}
\usepackage[tbtags]{amsmath}
\usepackage{amssymb}

\usepackage{amsmath,amsfonts,amsthm,amssymb}

\usepackage{indentfirst}
\usepackage{paralist}
\usepackage[english]{babel}
\usepackage{amsmath,amssymb}

\usepackage{cite}
\usepackage{tabularx}
\usepackage{booktabs}
\usepackage{dcolumn}

\usepackage{multirow}

\usepackage{graphicx}

\usepackage{wrapfig}
\usepackage{multicol}
\usepackage{verbatim}

\newtheoremstyle{theorem}
  {10pt}
  {10pt}
  {\sl}
 {}
  {\bf}
  {. }
  { }
  {}
\theoremstyle{theorem}

\newtheorem{theorem}{Theorem}[section]

\newtheorem{definition}{Definition}[section]
 \newtheorem{lemma}{Lemma}[section]
 
 \newtheorem{remark}{Remark}[section]

\numberwithin{equation}{section}

\newtheoremstyle{defi}
  {10pt}
  {10pt}
  {\rm}
  {}
  {\bf}
  {. }
  { }
  {}
\theoremstyle{defi}

\allowdisplaybreaks \textheight 21 true cm
\usepackage{epsfig,graphicx,epsf,wrapfig}
\usepackage{dcolumn}
\usepackage{bm}
\begin{document} 


\title{{\large  \textbf{Exponential Stability  of Solutions to   Stochastic Differential Equations Driven by $G$-L\'{e}vy Process\thanks{The work is supported in part by a NSFC Grant No. 11531006,
PAPD of Jiangsu Higher Education Institutions, and Jiangsu Center for Collaborative Innovation in Geographical Information Resource Development and Application.}   }}
\author{Bingjun Wang$^{1,2}$
\footnote{Email: wbj586@126.com} \  \  Hongjun Gao$^{1}$ \footnote{ Email: gaohj@njnu.edu.cn, gaohj@hotmail.com}
\\
\footnotesize{ 1. Jiangsu Provincial Key Laboratory for NSLSCS, School of Mathematical Science}
\\ \footnotesize{Nanjing Normal University, Nanjing 210023, P. R. China} 
\\
\footnotesize{2. Jinling Institute of Technology, Nanjing 211169, P. R. China}}
 }  

 \date{}
\maketitle 
\footnotesize
\noindent \textbf{Abstract~~~} In this paper,  BDG-type  inequality  for $G$-stochastic calculus with respect to $G$-L\'{e}vy process is obtained and solutions of  stochastic differential equations driven by  $G$-L\'{e}vy process under  non-Lipschitz   condition  are constructed.  Moreover, we establish the mean square exponential stability and quasi sure exponential stability of the solutions  be means of $G$-Lyapunov function method. An example is presented to illustrate the efficiency of the obtained results.
\\[2mm]
\textbf{Key words~~~}  $G$-L\'{e}vy Process; Non-Lipschitz; Exponential stability.
\\[2mm]
\\
\textbf{2010 Mathematics Subject Classification~~~}60H05, 60H10, 60H20

\section{Introduction}\label{s1}
In recent years much effort has been made to develop the theory of sublinear expectations connected with the volatility uncertainty and so-called $G$-Brownian motion.  $G$-Brownian motion was introduced by Shige Peng in \cite{S2} as a way to incorporate the unknown volatility into financial models. Its theory is tightly  associated  with the uncertainty problems involving an undominated family of probability measures.  Soon other connections have been discovered,  not only in the field of financial mathematics,  but also in the theory of path-dependent  partial differential equations or backward stochastic differential equations.  Thus $G$-Brownian motion and connected $G$-expectation are attractive mathematical objects.    We refer the reader to Gao \cite{F1}, Denis et al. \cite{L1}, Soner \cite{M1}, Bai et al. \cite{X2}, Li et al. \cite{X1},  Peng \cite{S3,S4,S5,S7,S8} and the references therein.

Returning however to the original problem of volatility uncertainty in the financial models, one  feels that $G$-Brownian motion is not  sufficient  to model the financial world, as both $G$- and the standard  Brownian motion share the same  property, which makes them often unsuitable for modelling, namely the continuity of paths.  Therefore, it is natural that Hu and Peng  \cite{M2} introduced the process with jumps, which they called $G$-L\'{e}vy process. Then Ren \cite{L2} introduced the representation of the sublinear expectation as an upper-expectation. In \cite{K1}, the author concentrated on establishing the integration theory for $G$-L\'{e}vy process with finite activity,   introduced the integral w.r.t the jump measure associated with the pure jump $G$-L\'{e}vy process and gave the It\^{o} formula for general $G$-It\^{o} L\'{e}vy process.

In \cite{F1}, the author proved the BDG inequality for $G$-stochastic calculus with respect to $G$-Brownian motion. In this article, we will prove the BDG-type inequality for $G$-stochastic calculus with respect to $G$-L\'{e}vy Process, which will be used in section 3.

In \cite{H1} and \cite{X2} the authors considered the stochastic differential equations driven by $G$-Brownian motion, where the coefficients do not satisfy the Lipschitz condition. Motivated by the aforementioned works, in section 3, the following stochastic differential equations driven by  $G$-L\'{e}vy process (GSDEs) is studied:
\begin{equation}
\begin{cases}
dY_{t}=b(t,Y_{t})dt+h_{ij}(t,Y_{t})d\langle B^{i},B^{j}\rangle_{t}+\sigma_{i}(t,Y_{t}) dB^{i}_{t}+\int_{R_{0}^{d}}K(t,Y_{t},z)L(dt,dz),
\\
Y_{t_{0}}=Y_{0},
\end{cases}
\end{equation}
 where $Y_{0}$ is the initial value with $\hat{\mathbb{E}}[|Y_{0}|^{2}]<\infty$,  $(\langle B^{i},B^{j}\rangle_{t})_{t\geq0}$ is the mutual variation process of the $d$-dimension $G$-Brownian motion $(B_{t})_{t\geq0}$, $L(\cdot,\cdot)$ is a Poisson random measure associated with the $G$-L\'{e}vy process $X$. The coefficients $b(\cdot,x),h_{ij}(\cdot,x),\sigma_{i}(\cdot,x)$ are in the space $M_{G}^{2}(0,T; R^{n})$, $K(\cdot,x,\cdot)\in H_{G}^{2}([0,T]\times R_{0}^{d}; R^{n})$ for each $x\in R^{n}$(these spaces will be defined in section 2).
If $b,h,\sigma $ and $K$ satisfy Lipschitz conditions,  Paczka \cite{K1} established the existence and uniqueness of solution for the stochastic differential equations (1.1) in the space $M_{G}^{2}(0,T;R^{n})$. However, many coefficients do not satisfy the Lipschitz condition. Therefore,  the extension to non-Lipschitz conditons is necessary.

Stability of stochastic differential equations has been well studied by many authors. In particular, by employing Lyapunov function method, Liu \cite{B1} obtained the comparison principles of $p$-th moment exponential stability for impulsive stochastic differential equations. For more details on this topic, one can see Peng and Jia \cite{S9}, Shen and Sun \cite{L4}, Wu and Han et al. \cite{S10}, Wu and Sun  \cite{H2}, Wu and Yan et al. \cite{X3} and the references therein. Very recently, Hu et al. \cite{L3} investigated the sufficient conditions for $p$-th moment stability of solutions to stochastic differential equations driven by  $G$-Brownian motion by means of a very special Lyapunov function. Zhang and Chen \cite{D1}, Fei and Fei \cite{W1} respectively established  the sufficient conditions for the exponential  stability and quasi sure exponential stability for a kind of special stochastic differential equations driven by  $G$-Brownian motion.  Then Ren  et al. \cite{Y1} established  the $p$-th  moment exponential stability and quasi sure exponential stability of  solutions  to impulsive stochastic differential equations driven by  $G$-Brownian motion.

To our best knowledge, there is no work reported on the mean square  exponential stability and quasi sure exponential stability of  solutions to GSDEs driven by $G$-L\'{e}vy process.  So, this should be developed. Motivated by the aforementioned works, the second part of this paper aims to establish the mean square  exponential stability and quasi sure exponential stability of  solutions to GSDEs driven by $G$-L\'{e}vy process by means of the $G$-Lyapunov function method.

The rest of this paper is organized as follows.  In section 2, we introduce some preliminaries and give the proof of BDG-type inequality for $G$-stochastic calculus with respect to $G$-jump measure.  In section 3, the solution of  (1.1) is constructed.    Section 4 is devoted to proving the mean square exponential stability and quasi sure   exponential stability of solutions to GSDEs driven by $G$-L\'{e}vy process by means of the $G$-Lyapunov function method. In the last section, an example is given to illustrate the effectiveness of the obtained results.

\section{Preliminaries} \label{s2}
In this section, we introduce some notations and preliminary results in $G$-framework which are needed in the following section. More details can be found in \cite{K1,K2,L2,M2,S1,S6}.

\begin{definition}
 Let $\Omega$ be a given set and let $\mathcal{H}$ be a linear space of real valued functions defined on $\Omega$. Moreover, if $X_{i}\in \mathcal{H},i=1,2,...,d,$  then $\varphi(X_{1},...,X_{d})\in \mathcal{H}$ for all $\varphi \in C_{b,lip}(R^{d}),$ where $C_{b,lip}(R^{d})$ is the space of all bounded real-valued Lipschitz  continuous functions.
 A sublinear expectation $\mathbb{E}$ is a functional $\mathbb{E}:\mathcal{H}\rightarrow R$ satisfying the following properties: for all $X,Y\in \mathcal{H}$, we have

 (i) Monotonicity: $\mathbb{E}[X]\geq \mathbb{E}[Y]$ if $X\geq Y.$

 (ii) Constant preserving: $\mathbb{E}[C]=C$ for $C\in R.$

 (iii) Sub-additivity:  $\mathbb{E}[X+Y]\leq \mathbb{E}[X]+\mathbb{E}[Y].$

 (iv) Positive homogeneity: $\mathbb{E}[\lambda X]=\lambda \mathbb{E}[X]$ for  $\lambda\geq 0.$
\end{definition}
The tripe $(\Omega,\mathcal{H},\mathbb{E})$ is called a sublinear expectation space.  $X\in \mathcal{H}$ is called a random variable in $(\Omega,\mathcal{H},\mathbb{E})$. We often call $Y=(Y_{1},...,Y_{d}), Y_{i}\in \mathcal{H}$ a $d$-dimensional random vector in $(\Omega,\mathcal{H},\mathbb{E}).$
\begin{definition}
In a sublinear expectation space $(\Omega,\mathcal{H},\mathbb{E})$, a $n$-dimensional random vector $Y=(Y_{1},...,Y_{n})$ is said to be independent from an $m$-dimensional  random vector $X=(X_{1},...,X_{m})$ if for each $\varphi \in C_{b,lip}(R^{m+ n}),$
$$\mathbb{E}[\varphi (X,Y)]=\mathbb{E}[\mathbb{E}[\varphi(x,Y)]_{x=X}].$$
\end{definition}
\begin{definition}
Let $X_{1}, X_{2}$ be two $n$-dimensional random vectors defined on a  sublinear expectation space $(\Omega_{1},\mathcal{H}_{1},\mathbb{E}_{1})$ and $(\Omega_{2},\mathcal{H}_{2},\mathbb{E}_{2})$, respectively. They are called identically distributed, denoted by $X_{1}\overset{d}{=}X_{2}$  if $$\mathbb{E}_{1}[\varphi(X_{1})]=\mathbb{E}_{2}[\varphi(X_{2})],  \quad  \forall\varphi \in C_{b,lip}(R^{n}).$$
$\bar{X}$ is said to be an independent copy of  $X$, if $\bar{X}$ is identically distributed with $X$ and independent of $X$.
\end{definition}
\begin{definition}($G$-L\'{e}vy process).
Let $X=(X_{t})_{t\geq0}$ be a $d-$dimensional c\`{a}dl\`{a}g process on a sublinear expectation space $(\Omega,\mathcal{H},\mathbb{E})$. We say that $X$ is a L\'{e}vy process if :

(i) $X_{0}=0,$

(ii) for each $s,t\geq 0$ the increment $X_{t+s}-X_{s}$ is independent of $(X_{t_{1}},\cdot\cdot\cdot, X_{t_{n}})$  for every $n\in N$ and every partition $0\leq t_{1}\leq t_{2}\leq\cdot\cdot\cdot\leq t_{n}\leq s,$

(iii) the distribution of the increment $X_{t+s}-X_{s}$,  $s, t\geq0$ is stationary, i.e. does not depend on $s.$

Moreover, we say that a L\'{e}vy process $X$ is a $G$-L\'{e}vy process,  if it satisfies additionally following conditions

(iv) there is a $2d$-dimensional L\'{e}vy process $(X_{t}^{c},X_{t}^{d})_{t\geq0}$ such for each  $t\geq 0 $,  $X_{t}=X_{t}^{c}+X_{t}^{d}$,

(v) process $X_{t}^{c}$ and $X_{t}^{d}$ satisfy the following conditions
\begin{align}
& \lim_{t\downarrow0}\mathbb{E}[|X_{t}^{c}|^{3}]t^{-1}=0; \hskip0.3cm \mathbb{E}[|X_{t}^{d}|]< Ct \hskip0.1cm for \hskip0.1cm all \hskip0.1cm t\geq0. \nonumber\end{align}
\end{definition}
\begin{remark}
The conditon (v) implies that $X^{c}$ is a $d$-dimensional generalized $G$-Brownian motion, whereas the jump part $X^{d}$ is of finite variation.(See \cite{M2} for details.)
\end{remark}
 Hu and Peng \cite{M2} noticed in their paper  that each $G$-L\'{e}vy process might be charactereized by a non-local operator $G_{X}.$

\begin{theorem}(\cite{M2})
  Let $X$  be a $G$-L\'{e}vy process in $R^{d}$. For every $f\in  C_{b}^{3}(R^{d})$ such that $f(0)=0$ we put
  $$G_{X}[f(\cdot)]:=\lim_{\delta\downarrow0}\mathbb{E}[f(X_{\delta})]\delta^{-1}.$$
  The above limit exist. Moreover, $G_{X}$ has the following Levy-Khintchine representation
  $$G_{X}[f(\cdot)]=\sup_{(v,p,Q)\in \mathcal{U}}\{\int_{R_{0}^{d}}f(z)v(dz)+\langle Df(0),p\rangle+\frac{1}{2}tr[D^{2}f(0)QQ^{T}]\},$$
  where $R_{0}^{d}:=R^{d}\backslash \{0\}$, $\mathcal{U}$ is a subset $\mathcal{U}\subset \mathcal{V}\times R^{d}\times \mathcal{Q}$ and $\mathcal{V}$ is a set of all Borel measures on $(R_{0}^{d}, \mathcal{B}(R_{0}^{d}))$.  $\mathcal{Q}$ is a set of all $d$-dimension positive definite symmetric matrix in  $\mathbb{S}^{d}$($\mathbb{S}^{d}$ is the space of all $d\times d$-dimensional symmetric matrices) such that
  \begin{align}
& \sup_{(v,p,Q)\in \mathcal{U}}\{\int_{R_{0}^{d}}|z|v(dz)+|p|+tr[QQ^{T}]\}<\infty.\end{align}

 \end{theorem}
 \begin{theorem}(\cite{M2})
 Let $X$ be a $d$-dimensional $G$-L\'{e}vy process. For each $\phi\in C_{b,lip}(R^{d})$, define $u(t,x):=\mathbb{E}[\phi(x+X_{t})].$  Then $u$ is the unique viscosity solution of the following integro-PDE
  \begin{align}
& 0=\partial_{t}u(t,x)-G_{X}[u(t,x+\cdot)-u(t,x)]\nonumber\\&=\partial_{t}u(t,x)-\sup_{(v,p,Q)\in \mathcal{U}}\{\int_{R_{0}^{d}}[u(t,x+z)-u(t,x)]v(dz)+\langle Du(t,x),p\rangle+\frac{1}{2}tr[D^{2}u(t,x)QQ^{T}]\},\end{align}
with initial condition $u(0,x)=\phi(x).$
 \end{theorem}
  \begin{theorem}
Let $\mathcal{U}$ satisfy (2.1). Consider the canonical space $\Omega:=\mathbb{D}_{0}(R^{+},R^{d})$ of all c\`{a}dl\`{a}g functions taking values in $R^{d}$ equipped with the Skorohod topology. Then there exists a sublinear expectations $\hat{\mathbb{E}}$ on $\mathbb{D}_{0}(R^{+},R^{d})$ such that the canonical process $(X_{t})_{t\geq0}$ is a $G$-L\'{e}vy process satisfying Levy-Khintchine representation with the same set $\mathcal{U}$.
  \end{theorem}

 The proof of above Theorem might be found in (Theorem 38 and 40 in \cite{M2}). We will give however the construction of $\hat{\mathbb{E}}$, as it is important to understand it.  We denote $\Omega_{T}:=\{w_{\cdot\wedge T}:w\in\Omega\}$. Put
  \begin{align}
&Lip(\Omega_{T}):=\{\xi\in L^{0}(\Omega_{T}):\xi=\phi(X_{t_{1}}, X_{t_{2}}-X_{t_{1}},...,X_{t_{n}}-X_{t_{n-1}}),\nonumber\\& \phi\in C_{b,lip}(R^{d\times n }), 0\leq t_{1}\leq\cdot\cdot\cdot\leq t_{n}\leq T\},\nonumber\end{align}
where $X_{t}(w)=w_{t}$ is the canonical process on the space $\mathbb{D}_{0}(R^{+},R^{d})$ and $L^{0}(\Omega)$ is the space of all random variables, which are measurable to the filtration generated by the canonical process. We also set
$$Lip(\Omega):=\bigcup_{T=1}^{\infty}Lip(\Omega_{T}).$$
Firstly, consider the random variable $\xi=\phi(X_{t+s}-X_{s}), \phi\in  C_{b,lip}(R^{d})$. We define
$$\hat{\mathbb{E}}[\xi]:=u(s,0),$$
where $u$ is a unique viscosity solution of integro-PDE (2.2) with the initial condition $u(0,x)=\phi(x).$  For general
$$\xi=\phi(X_{t_{1}},X_{t_{2}}-X_{t_{1}},...,X_{t_{n}}-X_{t_{n-1}}), \phi \in C_{b,Lip}(R^{d\times n})$$
we set $\hat{\mathbb{E}}[\xi]:=\phi_{n}$, where $\phi_{n}$ is obtained via the following iterated procedure
$$\phi_{1}(x_{1},...,x_{n-1})=\hat{\mathbb{E}}[\phi(x_{1},...,x_{n-1},X_{t_{n}}-X_{t_{n-1}})],$$
$$\phi_{2}(x_{1},...,x_{n-2})=\hat{\mathbb{E}}[\phi_{1}(x_{1},...,x_{n-2},X_{t_{n-1}}-X_{t_{n-2}})],$$
$$\vdots$$
$$\phi_{n-1}(x_{1})=\hat{\mathbb{E}}[\phi_{n-1}(x_{1},X_{t_{2}}-X_{t_{1}})],$$
$$\phi_{n}=\hat{\mathbb{E}}[\phi_{n-1}(X_{t_{1}})].$$
Lastly, we extend definition of $\hat{\mathbb{E}}$ on the completion of $Lip(\Omega_{T})$ (respectively $Lip(\Omega)$) under the norm $\|\cdot\|_{p}^{p}=\hat{\mathbb{E}}[|\cdot|^{p}], p\geq1.$  We denote such a completion by $L_{G}^{p}(\Omega_{T})$ (or resp. $L_{G}^{p}(\Omega)$).

Let $\mathcal{G}_{B}$ denote the set of all Borel function $g:R^{d}\rightarrow R^{d}$ such that $g(0)=0.$  Assume that for all L\'{e}vy  measure $\mu$ and $v\in \mathcal{V}$  there exist a function $g_{v}\in \mathcal{G}_{B}$ such that $v(B)=\mu (g_{v}^{-1}(B)),$ $\forall B\in \mathcal{B}(R_{0}^{d})$.  Then we can consider a different parametrizing set in the  L\'{e}vy-Khintchine formula. Namely using
$$\tilde{\mathcal{U}}:=\{(g_{v},p,Q)\in \mathcal{G}_{B}\times R^{d}\times \mathcal{Q}:(v,p,Q)\in \mathcal{U}\}.$$
It is elementary that the equation (2.2) is equivalent to the following equation
  \begin{align}
& 0=\partial_{t}u(t,x)-G_{X}[u(t,x+\cdot)-u(t,x)]\nonumber\\&=\partial_{t}u(t,x)-\sup_{(g,p,Q)\in \tilde{\mathcal{U}}}\{\int_{R_{0}^{d}}[u(t,x+g(z))-u(t,x)]\mu(dz)+\langle Du(t,x),p\rangle+\frac{1}{2}tr[D^{2}u(t,x)QQ^{T}]\}.\end{align}

Let $(\tilde{\Omega},\mathcal{F},P_{0})$  be a probability space carrying a Brownian motion $W$ and a L\'{e}vy process with a L\'{e}vy triplet $(0,0, \mu)$, which is independent of $W$. Let $N(\cdot,\cdot)$ be a Poisson random measure associated  with that L\'{e}vy process. Define $N_{t}=\int_{R_{0}^{d}}xN(t,dx)$, which is finite $P_{0}$-a.s. as we assume that $\mu$ integrates $|x|$. Moreover, in the finite activity case, i.e. $\lambda=\sup_{v\in \mathcal{V}}v(R_{0}^{d})<\infty$,  we define the Poisson process $M$ with intensity $\lambda$ by putting $M_{t}=N(t,R_{0}^{d})$.  We also define the filtration generated by $W$ and $N$:
$$\mathcal{F}_{t}:=\sigma\{W_{s},N_{s}:0\leq s\leq t\}\vee\mathcal{N}; \quad  \mathcal{N}:=\{A\in \tilde{\Omega}:P_{0}(A)=0\};  \quad \mathcal{F}:=(\mathcal{F}_{t})_{t\geq 0}.$$

\begin{theorem}(\cite{K1})
Introduce a set of integrands $\mathcal{A}_{t,T}^{\mathcal{U}}, 0\leq t\leq T$, associated with $\mathcal{U}$  as s set of all processes $\theta =(\theta^{d},\theta^{1,c},\theta^{2,c})$ defined on $]t,T]$ satisfying the following properties:

1.  $(\theta^{1,c}, \theta^{2,c} )$ is $ \mathcal{F}$-adapted process and $\theta^{d}$ is $ \mathcal{F}$-predictable  random field on $]t,T]\times R^{d}$.

2.  For $P_{0}$-a.a. $w\in \tilde{\Omega}$ and a.e. $s\in ]t,T]$ we have that $(\theta^{d}(s,\cdot)(w),\theta^{1,c}_{s}(w),\theta^{2,c}_{s}(w))\in \tilde{\mathcal{U}}$.

3.  $\theta$ satisfies the following integrability condition
$$E^{P_{0}}[\int_{t}^{T}(|\theta^{1,c}_{s}|+|\theta^{2,c}_{s}|^{2}+\int_{R_{0}^{d}}|\theta^{d}(s,z)|\mu(dz))ds]<\infty.$$
 For $\theta\in \mathcal{A}_{0,\infty}^{\mathcal{U}}$ denote the following L\'{e}vy-It\^{o} integral as
 $$B^{t,\theta}_{T}=\int_{t}^{T}\theta^{1,c}_{s} ds+ \int_{t}^{T}\theta^{2,c}_{s} dW_{s}+\int_{t}^{T}\int_{R_{0}^{d}}\theta^{d}(s,z)N(ds,dz).$$
For every $\xi=\phi(X_{t_{1}}, X_{t_{2}}-X_{t_{1}},X_{t_{n}}-X_{t_{n-1}})\in Lip(\Omega_{T})$, then $\hat{\mathbb{E}}[\xi]=\sup_{\theta\in\mathcal{A}_{0,\infty}^{\mathcal{U}}}E^{P_{0}}[\phi(B^{0,\theta}_{t_{1}}, B^{t_{1},\theta}_{t_{2}},\\\ldots , B^{t_{n-1},\theta}_{t_{n}})]$. Let $\xi\in L_{G}^{1}(\Omega)$, we can represent the sublinear expectation in the following way
$$\hat{\mathbb{E}}[\xi]=\sup_{\theta\in\mathcal{A}_{0,\infty}^{\mathcal{U}}}E^{P^{\theta}}[\xi],$$
where $P^{\theta}:=P_{0}\circ (B^{0,\theta}_{\cdot})^{-1}, \theta\in\mathcal{A}_{0,\infty}^{\mathcal{U}}$. We also denote $\mathfrak{B}:=\{P^{\theta}:\theta\in\mathcal{A}_{0,\infty}^{\mathcal{U}}\}$.
 \end{theorem}

\begin{definition}
We define the capacity $c$ associated with $\hat{\mathbb{E}}$ by putting
$$c(A):=\sup_{P\in \mathfrak{B}}P(A), \quad A\in \mathcal{B}(\Omega).$$
We will say that a set $A\in \mathcal{B}(\Omega)$ is polar if $c(A)=0$. We say that a property holds quasi-surely (q.s.) if it holds outside a polar set.
 \end{definition}

\begin{lemma}
Let $X\in L_{G}^{1}(\Omega_{T})$  and for some $p>0, \hat{\mathbb{E}}[|X|^{p}]<\infty$. Then, for each $M>0$,
$$c(|X|>M)\leq \frac{\hat{\mathbb{E}}[|X|^{p}]}{M^{p}}.$$
 \end{lemma}

Assume that $G$-L\'{e}vy process $X$ has finite activity, i.e. $$\lambda:=\sup_{v\in \mathcal{V}}v(R_{0}^{d})<\infty.$$  Without loss of generality we will also assume that $\lambda=1$ and that also $\mu(R_{0}^{d})=1$.   Let $X_{u-}$ denote the left limit of $X$ at point $u$, $\triangle X_{u}=X_{u}-X_{u-}$,
then we can define a Poisson random measure $L(ds,dz)$ associated with the $G$-L\'{e}vy process $X$ by putting
$$L(]s,t],A)=\sum_{s<u\leq t}\mathbb{I}_{A}(\Delta X_{u}),  \quad q.s.$$
for any $0<s<t<\infty$ and $A\in \mathcal{B}(R_{0}^{d})$.  The random measure is well-defined  and may be used to define the pathwise integral.

Let $H_{G}^{S}([0,T]\times R_{0}^{d})$ be a space of all elementary random fields on $[0,T]\times R_{0}^{d}$ of the form
$$K(r,z)(w)=\sum_{k=1}^{n-1}\sum_{l=1}^{m}F_{k,l}(w)\mathbb{I}_{]t_{k},t_{k+1}]}(r)\psi_{l}(z),  \quad n,m\in\mathbb{ N},$$
where $0\leq t_{1}<\ldots<t_{n}\leq T$ is the partition of $[0,T]$, $\{\psi_{l}\}_{l=1}^{m}\subset C_{b,lip}(R^{d})$ are functions with disjoint supports s.t. $\psi_{l}(0)=0$ and $F_{k,l}=\phi_{k,l}(X_{t_{1}},\ldots,X_{t_{k}}-X_{t_{k-1}})$, $\phi_{k,l}\in C_{b,lip}(R^{d\times k})$. We introduce the norm on this space
$$\|K\|^{p}_{H_{G}^{p}([0,T] \times R_{0}^{d})}:=\hat{\mathbb{E}} [\int_{0}^{T}\sup_{v\in\mathcal{V}}\int _{R_{0}^{d}}|K(r,z)|^{p}v(dz)dr],\quad p=1,2.$$

\begin{definition}
Let $0\leq s<t\leq T.$ The It\^{o} integral of $K\in H_{G}^{S}([0,T]\times R_{0}^{d})$ w.r.t. jump measure $L$ is defined as
$$\int_{s}^{t}\int_{R_{0}^{d}}K(r,z)L(dr,dz):=\sum_{s<r\leq t}K(r,\triangle X_{r}), \quad q.s.$$
 \end{definition}
\begin{lemma}
For every $K\in H_{G}^{S}([0,T]\times R_{0}^{d})$, we have that $\int_{0}^{T}\int_{R_{0}^{d}}K(r,z)L(dr,dz)$ is an element of $L_{G}^{2}(\Omega_{T}).$
\end{lemma}

Let $H_{G}^{p}([0,T]\times R_{0}^{d})$ denote the topological completion of $H_{G}^{S}([0,T]\times R_{0}^{d})$ under the norm $\|\cdot\|_{H_{G}^{p}([0,T]\times R_{0}^{d})},   p=1,2$.
Then It\^{o} integral can be continuously extended to the whole space $H_{G}^{p}([0,T]\times R_{0}^{d}), p=1,2$.   Moreover,  the extended  integral takes value in $L_{G}^{p}(\Omega_{T}), p=1,2$.

We now give the following BDG-type inequality for the integral defined above.
\begin{lemma}
For  $K(r,z)\in H_{G}^{2}([0,T]\times R_{0}^{d})$,  set $Y_{t}:=\int_{0}^{t}\int_{R_{0}^{d}}K(r,z)L(dr,dz)$.  Then there exists a c\`{a}dl\`{a}g modification  $\tilde{Y}_{t}$  of  $Y_{t}$ for all $t\in [0,T]$  such that
$$\hat{\mathbb{E}}[\sup_{0\leq t\leq T}|\tilde{Y_{t}}|^{2}]\leq C_{T}\hat{\mathbb{E}}[\int_{0}^{T}\sup_{v\in \mathcal{V}}\int_{R_{0}^{d}}K^{2}(r,z)v(dz)dr],$$
where $C_{T}>0$ is a constant depend on $T$.
\end{lemma}

\begin{proof}
Firstly, let us consider the case:
$$K(r,z)(w)=\sum_{k=1}^{n-1}\sum_{l=1}^{m}F_{k,l}(w)\mathbb{I}_{]t_{k},t_{k+1}]}(r)\psi_{l}(z)\in H_{G}^{S}([0,T]\times R_{0}^{d}).$$
For this case, the proof is similar to the theorem 27 in \cite{K1}. However, for completeness, we prove it as follows.
By the definition of the It\^{o}  integral  and Theorem 2.4 we have
 \begin{align}
&\hat{\mathbb{E}}[\sup_{0\leq t\leq T}(\int_{0}^{t}\int_{R_{0}^{d}}K(r,z)L(dr,dz))^{2}]=\sup _{\theta\in \mathcal{A}_{0,T}^{\mathcal{U}}}E^{P^{\theta}}[\sup_{0\leq t\leq T}(\sum_{0\leq r\leq t}K(r,\triangle X_{r}))^{2}]\nonumber\\
&=\sup _{\theta\in \mathcal{A}_{0,T}^{\mathcal{U}}}E^{P^{\theta}}[\sup_{0\leq t\leq T}(\sum_{0\leq r\leq t}\sum_{k=1}^{n-1}\sum_{l=1}^{m}\phi_{k,l}(X_{t_{1}\wedge t},\ldots,X_{t_{k}\wedge t}-X_{t_{k-1}\wedge t})\mathbb{I}_{]t_{k}\wedge t,t_{k+1}\wedge t]}(r)\psi_{l}(\triangle X_{r}))^{2}]
\nonumber\\
&=\sup _{\theta\in \mathcal{A}_{0,T}^{\mathcal{U}}}E^{P_{0}}[\sup_{0\leq t\leq T}(\sum_{0\leq r\leq t}\sum_{k=1}^{n-1}\sum_{l=1}^{m}\phi_{k,l}(B_{t_{1}\wedge t}^{0,\theta},\ldots,B_{t_{k}\wedge t}^{t_{k-1}\wedge t,\theta})\mathbb{I}_{]t_{k}\wedge t,t_{k+1}\wedge t]}(r)\psi_{l}(\triangle B_{r}^{0,\theta}))^{2}]
\nonumber\\
&=\sup _{\theta\in \mathcal{A}_{0,T}^{\mathcal{U}}}E^{P_{0}}[\sup_{0\leq t\leq T}(\sum_{0\leq r\leq t}\sum_{k=1}^{n-1}\sum_{l=1}^{m}F_{k,l}^{\theta}\mathbb{I}_{]t_{k}\wedge t,t_{k+1}\wedge t]}(r)\psi_{l}(\theta^{d}(r,\triangle N_{r})))^{2}],
\end{align}
where $F_{k,l}^{\theta}:=\phi_{k,l}(B_{t_{1}\wedge t}^{0,\theta},\ldots,B_{t_{k}\wedge t}^{t_{k-1}\wedge t,\theta})$  and $N_{t}$ is the Poisson process in Theorem 2.4.  Define a predictable process $K^{\theta}(r,z)$ as
$$K^{\theta}(r,z):=\sum_{k=1}^{n-1}\sum_{l=1}^{m}F_{k,l}^{\theta}\mathbb{I}_{]t_{k}\wedge t,t_{k+1}\wedge t]}(r)\psi_{l}(\theta^{d}(r,z)).$$
Then we can rewrite (2.4) as
 \begin{align}
&\hat{\mathbb{E}}[\sup_{0\leq t\leq T}(\int_{0}^{t}\int_{R_{0}^{d}}K(r,z)L(dr,dz))^{2}]=\sup _{\theta\in \mathcal{A}_{0,T}^{\mathcal{U}}}E^{P_{0}}[\sup_{0\leq t\leq T}(\int_{0}^{t}\int_{R_{0}^{d}}K^{\theta}(r,z)N(dr,dz))^{2}]\nonumber\\
&=\sup _{\theta\in \mathcal{A}_{0,T}^{\mathcal{U}}}E^{P_{0}}[\sup_{0\leq t\leq T}(\int_{0}^{t}\int_{R_{0}^{d}}K^{\theta}(r,z)\tilde{N}(dr,dz)+\int_{0}^{t}\int_{R_{0}^{d}}K^{\theta}(r,z)\mu(dz)dr)^{2}],
\end{align}
where $N(dr,dz)$ and $\tilde{N}(dr,dz)$ are respectively the Poisson random measure and the compensated Poisson measure associated with the L\'{e}vy process with the L\'{e}vy triplet $(0,0,\mu)$. Using the standard BDG inequality and H\"{o}lder inequality we get:
 \begin{align}
&\hat{\mathbb{E}}[\sup_{0\leq t\leq T}(\int_{0}^{t}\int_{R_{0}^{d}}K(r,z)L(dr,dz))^{2}]\nonumber\\
&\leq2\sup _{\theta\in \mathcal{A}_{0,T}^{\mathcal{U}}}E^{P_{0}}\{\sup_{0\leq t\leq T}[(\int_{0}^{t}\int_{R_{0}^{d}}K^{\theta}(r,z)\tilde{N}(dr,dz))^{2}+(\int_{0}^{t}\int_{R_{0}^{d}}K^{\theta}(r,z)\mu(dz)dr)^{2}]\}
\nonumber\\
&\leq2\sup _{\theta\in \mathcal{A}_{0,T}^{\mathcal{U}}}E^{P_{0}}\{\sup_{0\leq t\leq T}[(\int_{0}^{t}\int_{R_{0}^{d}}K^{\theta}(r,z)\tilde{N}(dr,dz))^{2}+t\int_{0}^{t}\int_{R_{0}^{d}}(K^{\theta}(r,z))^{2}\mu(dz)dr]\}
\nonumber\\
&\leq2\sup _{\theta\in \mathcal{A}_{0,T}^{\mathcal{U}}}\{E^{P_{0}}\sup_{0\leq t\leq T}(\int_{0}^{t}\int_{R_{0}^{d}}K^{\theta}(r,z)\tilde{N}(dr,dz))^{2}+E^{P_{0}}T\int_{0}^{T}\int_{R_{0}^{d}}(K^{\theta}(r,z))^{2}\mu(dz)dr\}
\nonumber\\
&\leq2\sup _{\theta\in \mathcal{A}_{0,T}^{\mathcal{U}}}\{\tilde{C}_{T}E^{P_{0}}\int_{0}^{T}\int_{R_{0}^{d}}(K^{\theta}(r,z))^{2}\mu(dz)dr+E^{P_{0}}T\int_{0}^{T}\int_{R_{0}^{d}}(K^{\theta}(r,z))^{2}\mu(dz)dr\}
\nonumber\\
&=C_{T}\sup _{\theta\in \mathcal{A}_{0,T}^{\mathcal{U}}}\int_{0}^{T}\int_{R_{0}^{d}}E^{P_{0}}(\sum_{k=1}^{n-1}\sum_{l=1}^{m}F_{k,l}^{\theta}\mathbb{I}_{]t_{k},t_{k+1}]}(r)\psi_{l}(\theta^{d}(r,z)))^{2}\mu(dz)dr,
\end{align}
where $\tilde{C}_{T}$ is a constant depend on $T$ and $C_{T}=2(T+\tilde{C}_{T})$. Note that the intervals $]t_{k},t_{k+1}]$ are mutually disjoint, hence
\begin{align}
&\hat{\mathbb{E}}[\sup_{0\leq t\leq T}(\int_{0}^{t}\int_{R_{0}^{d}}K(r,z)L(dr,dz))^{2}]
\nonumber\\
&\leq C_{T}\sup _{\theta\in \mathcal{A}_{0,T}^{\mathcal{U}}}\sum_{k=1}^{n-1}\sum_{l=1}^{m}\int_{t_{k}}^{t_{k+1}}E^{P_{0}}[(F_{k,l}^{\theta})^{2}\int_{R_{0}^{d}}\psi_{l}^{2}(\theta^{d}(r,z))\mu(dz)]dr
\nonumber\\
&= C_{T}\sup _{\theta\in \mathcal{A}_{0,T}^{\mathcal{U}}}\sum_{k=1}^{n-1}\int_{t_{k}}^{t_{k+1}}E^{P_{0}}[\sum_{l=1}^{m}\phi_{k,l}^{2}(B_{t_{1}}^{0,\theta},\ldots,B_{t_{k}}^{t_{k-1},\theta})\int_{R_{0}^{d}}\psi_{l}^{2}(\theta^{d}(r,z))\mu(dz)]dr
.\end{align}

By the assumptions on the process $\theta^{d}$, we know that for a.a. $w$ and a.e. $r$ function $z\rightarrow \theta^{d}(r,z)(w)$ is equal to $g_{v}$ for  $v\in \mathcal{V}$.  Hence we can transform  (2.7) to get
\begin{align}
&\hat{\mathbb{E}}[\sup_{0\leq t\leq T}(\int_{0}^{t}\int_{R_{0}^{d}}K(r,z)L(dr,dz))^{2}]
\nonumber\\
&\leq C_{T}\sup _{\theta\in \mathcal{A}_{0,T}^{\mathcal{U}}}\sum_{k=1}^{n-1}\int_{t_{k}}^{t_{k+1}}E^{P_{0}}[\sum_{l=1}^{m}\phi_{k,l}^{2}(B_{t_{1}}^{0,\theta},\ldots,B_{t_{k}}^{t_{k-1},\theta})\int_{R_{0}^{d}}\psi_{l}^{2}(z)v(dz)]dr
\nonumber\\
&\leq C_{T}\sup _{\theta\in \mathcal{A}_{0,T}^{\mathcal{U}}}E^{P_{0}}[\sum_{k=1}^{n-1}\int_{t_{k}}^{t_{k+1}}\sup_{v\in \mathcal{V}}\sum_{l=1}^{m}\phi_{k,l}^{2}(B_{t_{1}}^{0,\theta},\ldots,B_{t_{k}}^{t_{k-1},\theta})\int_{R_{0}^{d}}\psi_{l}^{2}(z)v(dz)dr]
\nonumber\\
&\leq C_{T}\sup _{\theta\in \mathcal{A}_{0,T}^{\mathcal{U}}}E^{P_{0}}[\int_{0}^{T}\sup_{v\in \mathcal{V}}\int_{R_{0}^{d}}\sum_{k=1}^{n-1}\sum_{l=1}^{m}\phi_{k,l}^{2}(B_{t_{1}}^{0,\theta},\ldots,B_{t_{k}}^{t_{k-1},\theta})\mathbb{I}_{]t_{k},t_{k+1}]}(r)\psi_{l}^{2}(z)v(dz)dr]
\nonumber\\
&= C_{T}\hat{\mathbb{E}}[\int_{0}^{T}\sup_{v\in \mathcal{V}}\int_{R_{0}^{d}}K^{2}(r,z)v(dz)dr].
\end{align}

For general $K(r,z)\in H_{G}^{2}([0,T]\times R_{0}^{d})$, choose $\{K^{n},n\geq1\} \subset H_{G}^{S}([0,T]\times R_{0}^{d})$ such that
$$\|K-K^{n}\|_{H_{G}^{2}([0,T] \times R_{0}^{d})}\rightarrow 0     \quad as \quad n\rightarrow\infty.$$
Set $Y_{t}^{n}=\int_{0}^{t}\int_{R_{0}^{d}}K^{n}(r,z)L(dr,dz)$. Then as $n,m\rightarrow\infty$,
$$\hat{\mathbb{E}}[\sup_{0\leq t\leq T}(Y_{t}^{n}-Y_{t}^{m})^{2}]\leq C_{T}\|K^{n}-K^{m}\|^{2}_{H_{G}^{2}([0,T] \times R_{0}^{d})}\rightarrow 0$$
and so there exist a subsequence $\{Y_{t}^{n_{k}},k\geq 1\}$ such that for any $k\geq1$,
$$(\hat{\mathbb{E}}[\sup_{0\leq t\leq T}(Y_{t}^{n_{k+1}}-Y_{t}^{n_{k}})^{2}])^{\frac{1}{2}}\leq \frac{1}{2^{k}}.$$
Then
\begin{align}
&(\hat{\mathbb{E}}[\sum_{k=1}^{\infty}\sup_{0\leq t\leq T}(Y_{t}^{n_{k+1}}-Y_{t}^{n_{k}})]^{2})^{\frac{1}{2}}=\sup _{\theta\in \mathcal{A}_{0,T}^{\mathcal{U}}}(E^{P^{\theta}}(\sum_{k=1}^{\infty}\sup_{0\leq t\leq T}(Y_{t}^{n_{k+1}}-Y_{t}^{n_{k}}))^{2})^{\frac{1}{2}}\nonumber\\
&\leq \sup _{\theta\in \mathcal{A}_{0,T}^{\mathcal{U}}}\sum_{k=1}^{\infty}(E^{P^{\theta}}(\sup_{0\leq t\leq T}(Y_{t}^{n_{k+1}}-Y_{t}^{n_{k}}))^{2})^{\frac{1}{2}}\leq \sum_{k=1}^{\infty}(\hat{\mathbb{E}}[\sup_{0\leq t\leq T}(Y_{t}^{n_{k+1}}-Y_{t}^{n_{k}})^{2}])^{\frac{1}{2}}\nonumber\\
&\leq 1,
\end{align}
which implies
$$\sum_{k=1}^{\infty}\sup_{0\leq t\leq T}|Y_{t}^{n_{k+1}}-Y_{t}^{n_{k}}|< \infty,    \hskip0.1cm  \hskip0.1cm q.s.$$

Set $\tilde{Y}_{t}=Y_{t}^{n_{1}}+\sum_{k=1}^{\infty}(Y_{t}^{n_{k+1}}-Y_{t}^{n_{k}})$, then $\tilde{Y}_{t}$ is q.s. defined on $\Omega$ for all $t\in [0,T]$ and for q.s. $w$, $t\rightarrow \tilde{Y}_{t}(w)$ is c\`{a}dl\`{a}g. Moreover, $(\hat{\mathbb{E}}[\sup_{0\leq t\leq T}\tilde{Y}_{t}^{2}])^{\frac{1}{2}}<\infty$, and
\begin{align}
&(\hat{\mathbb{E}}[\sup_{0\leq t\leq T}|Y_{t}^{n_{k}}-\tilde{Y}_{t}|^{2}])^{\frac{1}{2}}\leq(\hat{\mathbb{E}}(\sum_{l=k}^{\infty}\sup_{0\leq t\leq T}|Y_{t}^{n_{l+1}}-Y_{t}^{n_{l}}|)^{2})^{\frac{1}{2}}\nonumber\\
&\leq\sup _{\theta\in \mathcal{A}_{0,T}^{\mathcal{U}}}(E^{P^{\theta}}(\sum_{l=k}^{\infty}\sup_{0\leq t\leq T}|Y_{t}^{n_{l+1}}-Y_{t}^{n_{l}}|)^{2})^{\frac{1}{2}}\nonumber\\
&\leq \sum_{l=k}^{\infty}(\hat{\mathbb{E}}\sup_{0\leq t\leq T}|Y_{t}^{n_{l+1}}-Y_{t}^{n_{l}}|^{2})^{\frac{1}{2}}\rightarrow 0    \hskip0.1cm  \hskip0.1cm \hskip0.1cm  \hskip0.1cm  as \hskip0.1cm  \hskip0.1cm k\rightarrow\infty.
\end{align}

On the other hand, by $\|K-K^{n_{k}}\|_{H_{G}^{2}([0,T] \times R_{0}^{d})}\rightarrow 0$, we have
\begin{align}
&|(\hat{\mathbb{E}}[\int_{0}^{T}\sup_{v\in \mathcal{V}}\int_{R_{0}^{d}}K^{2}(r,z)v(dz)dr])^{\frac{1}{2}}-(\hat{\mathbb{E}}[\int_{0}^{T}\sup_{v\in \mathcal{V}}\int_{R_{0}^{d}}|K^{n_{k}}(r,z)|^{2}v(dz)dr])^{\frac{1}{2}}|\nonumber\\
&\leq\{\hat{\mathbb{E}}[(\int_{0}^{T}\sup_{v\in \mathcal{V}}\int_{R_{0}^{d}}K^{2}(r,z)v(dz)dr)^{\frac{1}{2}}-(\int_{0}^{T}\sup_{v\in \mathcal{V}}\int_{R_{0}^{d}}|K^{n_{k}}(r,z)|^{2}v(dz)dr)^{\frac{1}{2}}]^{2}\}^{\frac{1}{2}}\nonumber\\
& \leq(\hat{\mathbb{E}}[\int_{0}^{T}\sup_{v\in \mathcal{V}}\int_{R_{0}^{d}}|K(r,z)-K^{n_{k}}(r,z)|^{2}v(dz)dr])^{\frac{1}{2}}\rightarrow 0
\end{align}
as $k\rightarrow\infty.$  Then  combining (2.8)-(2.11), we have
\begin{align}
&\hat{\mathbb{E}}[\sup_{0\leq t\leq T}|\tilde{Y}_{t}|^{2}]\leq \hat{\mathbb{E}}[\sup_{0\leq t\leq T}|\tilde{Y}_{t}-Y_{t}^{n_{k}}|^{2}]+\hat{\mathbb{E}}[\sup_{0\leq t\leq T}|Y_{t}^{n_{k}}|^{2}]\nonumber\\
&\leq C_{T}\hat{\mathbb{E}}[\int_{0}^{T}\sup_{v\in \mathcal{V}}\int_{R_{0}^{d}}K^{2}(r,z)v(dz)dr].
\end{align}

Finally, since for any $t\in [0,T]$,  $\hat{\mathbb{E}}[|Y_{t}^{n_{k}}-Y_{t}|^{2}]\rightarrow 0$, we have $\hat{\mathbb{E}}[|Y_{t}-\tilde{Y}_{t}|^{2}]= 0$, thus $\tilde{Y}$ is a c\`{a}dl\`{a}g modification of  $Y$.
\end{proof}

Next, we introduce the It\^{o} integral of $G$-Brownian motion.
We consider the following type of simple process: for a given partition $\pi_{T}=t_{0},t_{1},...,t_{N}$ of $[0,T],$  set
$$\eta_{t}(w)=\sum_{k=0}^{N-1}\xi_{k}(w)I_{[t_{k},t_{k+1})}(t),$$
where $\xi_{k} \in L_{G}^{p}(\Omega_{t_{k}}),k=0,1,...,N-1$ are given.  The collection of these processes is denoted by $M_{G}^{p,0}(0,T)$. Denote by $M_{G}^{p}(0,T)$ the completion of  $M_{G}^{p,0}(0,T)$ under the norm $$\|\eta\|_{M_{G}^{p}(0,T)}=[\int_{0}^{T}\hat{\mathbb{E}}[|\eta_{t}|^{p}]dt]^{\frac{1}{p}}.$$
For a process $\eta\in {M_{G}^{p}(0,T)}$ $(p\geq2)$ one can define the stochastic integral w.r.t.  $G$-Brownian motion $B_{t}$ denoted by $\int_{0}^{t}\eta_{s} d B_{s}$. Similarly for $\eta\in M_{G}^{p}(0,T)$， $(p\geq1)$ one can define integrals $\int_{0}^{t}\eta_{s}ds$ and $\int_{0}^{t}\eta_{s}d\langle B\rangle_{s}$ respectively, where $\langle B\rangle_{t}$ is the quadratic variation process of $G$-Brownian motion $B_{t}$.  Moreover,  all of these integrals belong to $L_{G}^{p}(\Omega)$ for $p\geq1$(See \cite{S6} for details).

The following two lemmas from \cite{F1} are the BDG-type inequalities for the $G$-stochastic integral with respect to $B_{t}$ and $\langle B\rangle_{t}$.
\begin{lemma}
For $p\geq 2$, $\eta\in M_{G}^{p}(0,T)$.  Then
$$\hat{\mathbb{E}}[\sup_{0\leq u\leq T}|\int_{0}^{u}\eta_{r}dB_{r}|^{p}]\leq C_{p}T^{\frac{p}{2}-1}\int_{0}^{T}\hat{\mathbb{E}}|\eta_{r}|^{p}dr,$$
where $C_{p}>0$ is a constant only dependent on $p$.
\end{lemma}
\begin{lemma}
For $p\geq 1$, $\eta\in M_{G}^{p}(0,T)$.  Then there exists a constant $C'_{p}>0$ such that
$$\hat{\mathbb{E}}[\sup_{0\leq u\leq T}|\int_{0}^{u}\eta_{r}d\langle B\rangle_{r}|^{p}]\leq C'_{p}T^{p-1}\int_{0}^{T}\hat{\mathbb{E}}|\eta_{r}|^{p}dr.$$

\end{lemma}
Let $ M_{G}^{p}([0,T]; R^{n})$ and $H_{G}^{p}([0,T]\times R_{0}^{d}; R^{n})$ be the space of $n$-dimension stochastic process with each element belong to $M_{G}^{p}(0,T)$ and $H_{G}^{p}([0,T]\times R_{0}^{d})$ respectively.

\section{SDEs Driven by $G$-L\'{e}vy Process}  \label{s3}
In this section, we consider the solution of the following $n$-dimension GSDEs:
\begin{equation}
\begin{cases}
dY_{t}=b(t,Y_{t})dt+h_{ij}(t,Y_{t})d\langle B^{i}, B^{j}\rangle_{t}+\sigma_{i}(t,Y_{t}) dB^{i}_{t}+\int_{R_{0}^{d}}K(t,Y_{t},z)L(dt,dz),
\\
Y_{t_{0}}=Y_{0},
\end{cases}
\end{equation}
where $b(\cdot, x), h_{ij}(\cdot, x),\sigma_{i}(\cdot, x)\in M_{G}^{2}([0,T]; R^{n}), K(\cdot, x, \cdot)\in H_{G}^{2}([0,T]\times R_{0}^{d}; R^{n})$ for each $x\in R^{n}$, $Y_{0}\in R^{n}$ is the initial value with $\hat{\mathbb{E}}|Y_{0}|^{2}<\infty$,
$(\langle B^{i}, B^{j}\rangle_{t})_{t\geq t_{0}}$ is the mutual variation process of the $d$-dimension  $G$-Brownian motion $(B_{t})_{t\geq t_{0}}$.

Here and in the rest of this paper we use the Einstein convention, i.e., the above repeated indices of $i$ and $j$ within one term imply the summation form 1 to $d$,  i.e.,
$$\int_{0}^{t}h_{ij}(s,Y_{s})d\langle B^{i}, B^{j}\rangle_{s}:=\sum_{i,j=1}^{d}\int_{0}^{t}h_{ij}(s,Y_{s})d\langle B^{i}, B^{j}\rangle_{s},$$
$$\int_{0}^{t}\sigma_{i}(s,Y_{s}) dB^{i}_{s}:=\sum_{i=1}^{d}\int_{0}^{t}\sigma_{i}(s,Y_{s}) dB^{i}_{s}.$$

\begin{theorem}
Suppose that

(a) there exists a function $H(t,u): R_{+}\times  R_{+}\rightarrow  R_{+}$ such that

(a1) for fixed $t$, $H(t,u)$ is continuous nondecreasing with respect to $u$,

(a2)  for $0\leq t_{0}<t\leq T$ and $X_{t}\in L_{G}^{2}(\Omega_{t})$,
$$b(t,X_{t}),h_{ij}(t,X_{t}),\sigma_{i}(t,X_{t})\in M_{G}^{2}(0,T;R^{n}), K(t,X_{t},z)\in H_{G}^{2}([0,T]\times R_{0}^{d};R^{n})$$
and
\begin{align}
&\hat{\mathbb{E}}[|b(t,X_{t})|^{2}]+\hat{\mathbb{E}}[|h_{ij}(t,X_{t})|^{2}]+\hat{\mathbb{E}}[|\sigma_{i}(t,X_{t})|^{2}]+\hat{\mathbb{E}}[\sup_{v\in \mathcal{V}}\int_{R_{0}^{d}}|K(t,X_{t},z)|^{2}v(dz)]\nonumber\\
&\leq H(t,\hat{\mathbb{E}}[\sup_{r\leq t}|X_{r}|^{2}]),\end{align}

(a3) for any $M>0$, the differential equation
$$\frac{du}{dt}=MH(t,u)$$
has a global solution $u_{t}$ for any initial value $u_{t_{0}}$;

(b) there exist a function $F(t,u): R_{+}\times  R_{+}\rightarrow  R_{+}$ such that

(b1) for fixed $t$, $F(t,u)$ is continuous nondecreasing in $u$ and $F(t,0)=0$,

(b2) for $t_{0}<t\leq T$ and $X_{t}, Y_{t}\in L_{G}^{2}(\Omega_{t})$,
\begin{align}
&\hat{\mathbb{E}}[|b(t,X_{t})-b(t,Y_{t})|^{2}]+\hat{\mathbb{E}}[|h_{ij}(t,X_{t})-h_{ij}(t,Y_{t})|^{2}]+\hat{\mathbb{E}}[|\sigma_{i}(t,X_{t})-\sigma_{i}(t,Y_{t})|^{2}]\nonumber\\
&+\hat{\mathbb{E}}[\sup_{v\in \mathcal{V}}\int_{R_{0}^{d}}|(K(t,X_{t},z)-K(t,Y_{t},z))|^{2}v(dz)]\leq F(t,\hat{\mathbb{E}}[\sup_{r\leq t}|X_{r}-Y_{r}|^{2}]),
\end{align}

(b3) for any constant $M>0$, if a non-negative function $\varphi_{t}$ satisfies
$$\varphi_{t}\leq M\int_{t_{0}}^{t}F(s,\varphi_{s})ds$$
for all $t> t_{0}$, then $\varphi_{t}=0$.

Then (3.1) has a unique  c\`{a}dl\`{a}g  solution  $Y_{t}\in L_{G}^{2}(\Omega_{t})$ for $t_{0}<t\leq T$.
\end{theorem}

{\bf Proof~~~}
Let $Y_{t}^{0}:=Y_{0}$ and for $n\in \mathbb{N}$,
\begin{align}
&Y_{t}^{n}:=Y_{0}+\int_{t_{0}}^{t}b(s,Y_{s}^{n-1})ds+\int_{t_{0}}^{t}h_{ij}(s,Y_{s}^{n-1})d\langle B^{i},B^{j}\rangle_{s}+\int_{t_{0}}^{t}\sigma_{i}(s,Y_{s}^{n-1})dB^{i}_{s}\nonumber\\
&+\int_{t_{0}}^{t}\int_{R_{0}^{d}}K(s,Y_{s}^{n-1},z)L(ds,dz).
\end{align}
First of all, we show that for $t_{0}<t\leq T$ and $n\in \mathbb{N}$,
\begin{align}
&Y_{t}^{n}\in L_{G}^{2}(\Omega_{t})\hskip0.1cm  \hskip0.1cm and  \hskip0.1cm  \hskip0.1cm \hat{\mathbb{E}}[\sup_{r\leq t}|Y_{r}^{n}|^{2}]\leq u_{t}\leq u_{T},\end{align}
where $u_{t}$ is the solution of differential equation in (a3) satisfies
$$u_{t}=C_{1}(T)\hat{\mathbb{E}}[|Y_{0}|^{2}]+C_{1}(T)\int_{t_{0}}^{t}H(s,u_{s})ds$$
and $$C_{1}(T):=5(1+T+C'_{2}T+C_{2}+C_{T}).$$

Suppose $Y_{t}^{n-1}\in L_{G}^{2}(\Omega_{t})$ and $\hat{\mathbb{E}}[\sup_{r\leq t}|Y_{r}^{n-1}|^{2}]\leq u_{t}$, which together with the definition of the $G$-stochastic integral and (a2) yield $Y_{t}^{n}\in L_{G}^{2}(\Omega_{t})$.

Secondly, by $C_{r}$-inequality,  Lemma 2.3-2.5, H\"{o}lder inequality and (a1)-(a2), we get
\begin{align}
&\hat{\mathbb{E}}[\sup_{r\leq t}|Y_{r}^{n}|^{2}]\leq 5\{\hat{\mathbb{E}}[|Y_{0}|^{2}]+\hat{\mathbb{E}}[\sup_{r\leq t}|\int_{t_{0}}^{r}b(s,Y_{s}^{n-1})ds|^{2}]+\hat{\mathbb{E}}[\sup_{r\leq t}|\int_{t_{0}}^{r}h_{ij}(s,X_{s}^{n-1})d\langle B^{i},B^{j}\rangle_{s}|^{2}]\nonumber\\
&+\hat{\mathbb{E}}[\sup_{r\leq t}|\int_{t_{0}}^{r}\sigma_{i}(s,Y_{s}^{n-1})dB^{i}_{s}|^{2}]+\hat{\mathbb{E}}[\sup_{r\leq t}|\int_{t_{0}}^{r}\int_{R_{0}^{d}}K(s,Y_{s}^{n-1},z)L(ds,dz)|^{2}]\}\nonumber\\
&\leq 5\{\hat{\mathbb{E}}[|Y_{0}|^{2}]+t\int_{t_{0}}^{t}\hat{\mathbb{E}}[|b(s,Y_{s}^{n-1})|^{2}]ds+C'_{2}t\int_{t_{0}}^{t}\hat{\mathbb{E}}[|h_{ij}(s,Y_{s}^{n-1})|^{2}]ds+C_{2}\int_{t_{0}}^{t}\hat{\mathbb{E}}[|\sigma_{i}(s,Y_{s}^{n-1})|^{2}]ds\nonumber\\
&+C_{t}\int_{t_{0}}^{t}\hat{\mathbb{E}}[\sup_{v\in \mathcal{V}}\int_{R_{0}^{d}}|K(s,Y_{s}^{n-1},z)|^{2}v(dz)]ds\}\nonumber\\
&\leq C_{1}(T)(\hat{\mathbb{E}}[|Y_{0}|^{2}]+\int_{t_{0}}^{t}H(s,\hat{\mathbb{E}}[\sup_{r\leq s}|Y_{r}^{n-1}|^{2}])ds)\nonumber\\
&\leq C_{1}(T)(\hat{\mathbb{E}}[|Y_{0}|^{2}]+\int_{t_{0}}^{t}H(s,u_{s})ds)  \leq u_{t}
\end{align}
for all $t\leq T$.  By the induction method, (3.5) is proved.

Next, by the same deduction as above, we have
\begin{align}
&\hat{\mathbb{E}}[\sup_{r\leq t}|Y_{r}^{n}-Y_{r}^{m}|^{2}]\leq 4\{\hat{\mathbb{E}}[\sup_{r\leq t}|\int_{t_{0}}^{r}(b(s,Y_{s}^{n-1})-b(s,Y_{s}^{m-1}))ds|^{2}]\nonumber\\
&+\hat{\mathbb{E}}[\sup_{r\leq t}|\int_{t_{0}}^{r}(h_{ij}(s,Y_{s}^{n-1})-h_{ij}(s,Y_{s}^{m-1}))d\langle B^{i},B^{j}\rangle_{s}|^{2}]+\hat{\mathbb{E}}[\sup_{r\leq t}|\int_{t_{0}}^{r}(\sigma_{i}(s,Y_{s}^{n-1})-\sigma_{i}(s,Y_{s}^{m-1}))dB^{i}_{s}|^{2}]\nonumber\\
&+\hat{\mathbb{E}}[\sup_{r\leq t}|\int_{t_{0}}^{r}\int_{R_{0}^{d}}(K(s,Y_{s}^{n-1},z)-K(s,Y_{s}^{m-1},z))L(ds,dz)|^{2}]\}\nonumber\\
&\leq C_{2}(T)\{\int_{t_{0}}^{t}\hat{\mathbb{E}}[|b(s,Y_{s}^{n-1})-b(s,Y_{s}^{m-1})|^{2}]ds\nonumber\\
&+\int_{t_{0}}^{t}\hat{\mathbb{E}}[|h_{ij}(s,Y_{s}^{n-1})-h_{ij}(s,Y_{s}^{m-1})|^{2}]ds+\int_{t_{0}}^{t}\hat{\mathbb{E}}[|\sigma_{i}(s,Y_{s}^{n-1})-\sigma_{i}(s,Y_{s}^{m-1})|^{2}]ds\nonumber\\
&+\int_{t_{0}}^{t}\hat{\mathbb{E}}[\sup_{v\in \mathcal{V}}\int_{R_{0}^{d}}|K(s,Y_{s}^{n-1},z)-K(s,Y_{s}^{m-1},z)|^{2}v(dz)]ds\}\nonumber\\
&\leq C_{2}(T)\int_{t_{0}}^{t}F(s,\hat{\mathbb{E}}[\sup_{r\leq s}|Y_{r}^{n-1}-Y_{r}^{m-1}|^{2}])ds,
\end{align}
where $C_{2}(T)=4(T+C_{2}'T+C_{2}+C_{T})$.

Let $$\xi_{t}=\limsup_{n,m\rightarrow\infty}\hat{\mathbb{E}}[\sup_{r\leq t}|Y_{r}^{n-1}-Y_{r}^{m-1}|^{2}]$$
It follows from the Fatou lemma and (b1) that
$$\xi_{t}\leq C_{2}(T)\int_{t_{0}}^{t}F(s,\xi_{s})ds.$$
By (b3), we obtain that $\xi_{t}=0$, i.e.
$$\limsup_{n,m\rightarrow\infty}\hat{\mathbb{E}}[\sup_{r\leq t}|Y_{r}^{n-1}-Y_{r}^{m-1}|^{2}]=0.$$
Then there exists a subsequence $Y_{t}^{n_{k}}$ such that for any $k\geq1$,
$$(\hat{\mathbb{E}}[\sup_{r\leq t}|Y_{r}^{n_{k+1}}-Y_{r}^{n_{k}}|^{2}])^{\frac{1}{2}}\leq \frac{1}{2^{k}}.$$
Thus
\begin{align}
&(\hat{\mathbb{E}}[\sum_{k=1}^{\infty}\sup_{r\leq t}|Y_{r}^{n_{k+1}}-Y_{r}^{n_{k}}|]^{2})^{\frac{1}{2}}=\sup _{\theta\in \mathcal{A}_{0,T}^{\mathcal{U}}}(E^{P^{\theta}}(\sum_{k=1}^{\infty}\sup_{r\leq t}|Y_{r}^{n_{k+1}}-Y_{r}^{n_{k}}|)^{2})^{\frac{1}{2}}\nonumber\\
&\leq \sup _{\theta\in \mathcal{A}_{0,T}^{\mathcal{U}}}\sum_{k=1}^{\infty}(E^{P^{\theta}}(\sup_{r\leq t}|Y_{r}^{n_{k+1}}-Y_{r}^{n_{k}}|)^{2})^{\frac{1}{2}}\leq \sum_{k=1}^{\infty}(\hat{\mathbb{E}}[\sup_{r\leq t}|Y_{r}^{n_{k+1}}-Y_{r}^{n_{k}}|^{2}])^{\frac{1}{2}}\nonumber\\
&\leq 1,
\end{align}
which implies
$$\sum_{k=1}^{\infty}\sup_{r\leq t}|Y_{r}^{n_{k+1}}-Y_{r}^{n_{k}}|< \infty    \hskip0.1cm  \hskip0.1cm q.s.$$
Set $Y_{t}=Y_{t}^{n_{1}}+\sum_{k=1}^{\infty}(Y_{t}^{n_{k+1}}-Y_{t}^{n_{k}})$, then $Y_{t}$ is q.s. defined on $\Omega$ for all $t\in [0,T]$ and c\`{a}dl\`{a}g. Moreover, $(\hat{\mathbb{E}}[\sup_{r\leq t}|Y_{r}|^{2}])^{\frac{1}{2}}<\infty$, and
\begin{align}
&(\hat{\mathbb{E}}[\sup_{r\leq t}|Y_{r}^{n_{k}}-Y_{r}|^{2}])^{\frac{1}{2}}\leq(\hat{\mathbb{E}}(\sum_{l=k}^{\infty}\sup_{r\leq t}|Y_{r}^{n_{l+1}}-Y_{r}^{n_{l}}|)^{2})^{\frac{1}{2}}\nonumber\\
&=\sup _{\theta\in \mathcal{A}_{0,T}^{\mathcal{U}}}(E^{P^{\theta}}(\sum_{l=k}^{\infty}\sup_{r\leq t}|Y_{r}^{n_{l+1}}-Y_{r}^{n_{l}}|)^{2})^{\frac{1}{2}}\leq \sum_{l=k}^{\infty}(\hat{\mathbb{E}}\sup_{r\leq t}|Y_{r}^{n_{l+1}}-Y_{r}^{n_{l}}|^{2})^{\frac{1}{2}}.
\end{align}
Letting $k\rightarrow \infty$ and taking limits on both sides of the above inequality, we get
$$\lim_{k\rightarrow\infty}\hat{\mathbb{E}}[\sup_{r\leq t}|Y_{r}^{n_{k}}-Y_{r}|^{2}]=0.$$
Then by the H\"{o}lder inequality, (b2) and Lemma 2.3-2.5, it holds that
\begin{align}
&\hat{\mathbb{E}}[\sup_{r\leq t}|\int_{t_{0}}^{r}b(s,Y_{s}^{n_{k}})ds-\int_{t_{0}}^{r}b(s,Y_{s})ds|^{2}]\leq C_{2}(T)\int_{t_{0}}^{t}F(s,\hat{\mathbb{E}}[\sup_{r\leq s}|Y_{r}^{n_{k}}-Y_{r}|^{2})ds,\nonumber
\end{align}
\begin{align}
&\hat{\mathbb{E}}[\sup_{r\leq t}|\int_{t_{0}}^{r}h_{ij}(s,Y_{s}^{n_{k}})d\langle B^{i},B^{j}\rangle_{s}-\int_{t_{0}}^{r}b(s,Y_{s})d\langle B^{i},B^{j}\rangle_{s}|^{2}]\leq C_{2}(T)\int_{t_{0}}^{t}F(s,\hat{\mathbb{E}}[\sup_{r\leq s}|Y_{r}^{n_{k}}-Y_{r}|^{2})ds,\nonumber
\end{align}
\begin{align}
&\hat{\mathbb{E}}[\sup_{r\leq t}|\int_{t_{0}}^{r}\sigma_{i}(s,Y_{s}^{n_{k}})dB^{i}_{s}-\int_{t_{0}}^{r}\sigma_{i}(s,Y_{s})dB^{i}_{s}|^{2}]\leq C_{2}(T)\int_{t_{0}}^{t}F(s,\hat{\mathbb{E}}[\sup_{r\leq s}|Y_{r}^{n_{k}}-Y_{r}|^{2})ds,\nonumber
\end{align}
and
\begin{align}
&\hat{\mathbb{E}}[\sup_{r\leq t}|\int_{t_{0}}^{r}\int_{R_{0}^{d}}K(s,Y_{s}^{n_{k}},z)L(ds,dz)-\int_{t_{0}}^{r}\int_{R_{0}^{d}}K(s,Y_{s},z)L(ds,dz)|^{2}]\nonumber\\
&\leq C_{2}(T)\int_{t_{0}}^{t}F(s,\hat{\mathbb{E}}[\sup_{r\leq s}|Y_{r}^{n_{k}}-Y_{r}|^{2})ds.\nonumber
\end{align}
Taking limits on both sides of (3.4) in $L_{G}^{2}(\Omega_{t})$, we obtain that $Y$ satisfies (3.1).

Next, let $Y$ and $Y'$ be both solutions of (1.1), then by the same way as above, we obtain that
$$\hat{\mathbb{E}}[\sup_{r\leq t}|Y_{r}-Y'_{r}|^{2}]\leq C_{2}(T)\int_{t_{0}}^{t}F(s,\hat{\mathbb{E}}[\sup_{r\leq s}|Y_{r}-Y'_{r}|^{2}])ds$$
for all $t\leq T$. We can apply (b3) deduce that $\hat{\mathbb{E}}[\sup_{r\leq t}|Y_{r}-Y'_{r}|^{2}]=0$, which implies that $Y_{t}=Y'_{t}, \hskip0.1cm  \hskip0.1cm  t_{0}<t \leq T$ q.s..  Thus the proof is completed.

\section{Exponential stability of the solutions }  \label{s3}
 In this section, we consider exponential stability of the following $n$-dimension GSDEs:
\begin{equation}
\begin{cases}
dY_{t}=b(t,Y_{t})dt+h_{ij}(t,Y_{t})d\langle B^{i}, B^{j}\rangle_{t}+\sigma_{i}(t,Y_{t}) dB^{i}_{t}+\int_{R_{0}^{d}}K(t,Y_{t},z)L(dt,dz),
\\
Y_{t_{0}}=Y_{0}
\end{cases}
\end{equation}
where $b, h_{ij},\sigma_{i}\in M_{G}^{2}([0,T]; R^{n}), K\in H_{G}^{2}([0,T]\times R_{0}^{d}; R^{n})$, $Y_{0}\in R^{n}$ is the initial value with $\hat{\mathbb{E}}|Y_{0}|^{2}<\infty$,
$(\langle B^{i}, B^{j}\rangle_{t})_{t\geq t_{0}}$ is the mutual variation process of the $d$-dimension  $G$-Brownian motion $(B_{t})_{t\geq t_{0}}$.  We assume the functions $b, h_{ij},\sigma_{j}$ and $K$ satisfy all necessary conditions for the global existence and uniqueness of solutions for all $t\geq t_{0}$. For the purpose of stability in this paper, we also assume that $b(t,0)=0, h_{ij}(t,0)=0, \sigma_{i}(t,0)=0, K(t, 0, z)=0$.  Thus, the system (4.1) has a trivial solution.

\begin{definition}
The trivial solution of the system (4.1) is said to be

(1) mean square exponential stable if for any initial $Y_{0}$, the solution $Y_{t}$ satisfies that
$$\hat{\mathbb{E}}|Y_{t}|^{2}\leq C\hat{\mathbb{E}}|Y_{0}|^{2}e^{-\lambda(t-t_{0})},$$
where $\lambda$ and $C$ are positive constants independent of $t_{0}$.

(2) quasi sure exponentially stable if  the solution $Y_{t}$ satisfies that
$$\limsup_{t\rightarrow\infty}\frac{1}{t}\ln |Y_{t}|\leq -\lambda,  \quad q.s.,$$
for any initial data $Y_{0}$ and $\lambda> 0.$
 \end{definition}
\begin{definition}
The function $V$ is said to belong to the class $v_{0}$, if $V(t,Y)\in C^{1,2}([t_{0},+\infty)\times R^{n},R^{+})$, i.e., $V_{t},V_{Y},V_{YY}$ are continuous on $[t_{0},+\infty)\times R^{n}$, and $V_{YY}$ satisfy local Lipschitz condition, where
$$V_{t}(t,Y):=\frac{\partial V(t,Y)}{\partial t},    \quad \quad V_{Y}(t,Y):=(\frac{\partial V(t,Y)}{\partial Y_{1}}, \frac{\partial V(t,Y)}{\partial Y_{2}},\ldots, \frac{\partial V(t,Y)}{\partial Y_{n}})$$
and
$$V_{YY}(t,Y):=(\frac{\partial ^{2}V(t,Y)}{\partial Y_{i}\partial Y_{j}})_{n\times n}.$$
 \end{definition}
\begin{definition}
For each $V\in v_{0}$, we define an operator $L$ by
\begin{align}
& LV(t, Y_{t}):=V_{t}(t,Y_{t})+\langle V_{Y}(t,Y_{t}),b(t,Y_{t})\rangle\nonumber\\
&+\sup_{Q\in \mathcal{Q}}tr[(\langle V_{Y}(t,Y_{t}),h(t,Y_{t})\rangle+\frac{1}{2}\langle V_{YY}(t,Y_{t})\sigma(t,Y_{t}),\sigma(t,Y_{t})\rangle)QQ^{T}]\nonumber\\
&+\sup_{v\in \mathcal{V}}\int_{R_{0}^{d}}(V(t,Y_{t^{-}}+K(t,Y_{t},z))-V(t,Y_{t^{-}}))v(dz),
\end{align}
where $\langle V_{Y}(t,Y_{t}),h(t,Y_{t})\rangle+\langle V_{YY}(t,Y)\sigma(t,Y_{t}),\sigma(t,Y_{t})\rangle$ is the symmetric matrix in $\mathbb{S}^{d}$,  with the form
 \begin{align}
& \langle V_{Y}(t,Y_{t}),h(t,Y_{t})\rangle+\langle V_{YY}(t,Y_{t})\sigma(t,Y_{t}),\sigma(t,Y_{t})\rangle\nonumber\\
&:=[\langle V_{Y}(t,Y_{t}),h_{ij}(t,Y_{t})\rangle+\langle V_{YY}(t,Y_{t})\sigma_{i}(t,Y_{t}),\sigma_{j}(t,Y_{t})\rangle]_{i,j=1}^{d}.\nonumber
\end{align}
 \end{definition}

 Let $Y_{t}$ be a solution of (4.1), for convention, we use the following notations in the sequel
 \begin{align}
& M_{t}^{s}:=\int_{s}^{t}e^{\lambda r}[\langle V_{Y}(r,Y_{r}),h_{ij}(r,Y_{r})\rangle +\frac{1}{2}\langle V_{YY}(r,Y_{r})\sigma_{i}(r,Y_{r}),\sigma_{j}(r,Y_{r})\rangle]d\langle B^{i},B^{j}\rangle_{r}\nonumber\\
&-\int_{s}^{t}e^{\lambda r}\sup_{Q\in \mathcal{Q}}tr[(\langle V_{Y}(t,Y_{t}),h(t,Y_{t})\rangle+\frac{1}{2}\langle V_{YY}(r,Y_{r})\sigma(r,Y_{r}),\sigma(r,Y_{r})\rangle)QQ^{T}]dr,\nonumber\\
&P_{t}^{s}=\int_{s}^{t}\int_{R_{0}^{d}}e^{\lambda r}[V(r,Y_{r^{-}}+K(r,Y_{r},z))-V(r,Y_{r^{-}})]L(dr,dz)\nonumber\\
&-\int_{s}^{t}\sup_{v\in \mathcal{V}}\int_{R_{0}^{d}}e^{\lambda r}[V(r,Y_{r^{-}}+K(r,Y_{r},z))-V(r,Y_{r^{-}})]v(dz)dr.
 \end{align}
From Theorem 2.2 in Peng \cite{S6},  $\{M_{t}^{s}\}_{t\geq s}$ is a $G$-martingale. From Theorem 13 in \cite{K2},  $P_{t}^{s}$ is also a  $G$-martingale.

We are now in a position to propose the mean square exponentially stability for the system (4.1).

\begin{theorem}
Assume that there exist a $V\in v_{0}$, constants $C_{4}>C_{3}>0$ and $\lambda> 0$ such that

(c)  $C_{3}|Y|^{2}\leq V(t,Y)\leq C_{4}|Y|^{2}$ for all $t\geq t_{0},Y\in R^{n},$

(d) $LV(t,Y_{t})\leq -\lambda V(t,Y_{t})$.

Then, the trivial solution of system (4.1) is mean square exponentially stable.
\end{theorem}

{\bf Proof~~~}
For $t\in [t_{0},T]$, applying the $G$-It\^{o} formula (Theorem 32 in \cite{K1}) to $e^{\lambda t}V(t,Y_{t})$, we obtain
\begin{align}
&d(e^{\lambda t}V(t,Y_{t}))=e^{\lambda t}[\lambda V(t,Y_{t})+V_{t}(t,Y_{t})+\langle V_{Y}(t,Y_{t}), b(t,Y_{t})\rangle]dt\nonumber\\
&+e^{\lambda t}\langle V_{Y}(t,Y_{t}),\sigma_{j}(t,Y_{t})\rangle dB_{t}^{i}+e^{\lambda t}\langle V_{Y}(t,Y_{t}),h_{ij}(t,Y_{t})\rangle d\langle B^{i},B^{j}\rangle_{t}\nonumber\\
&+\frac{1}{2}e^{\lambda t}\langle V_{YY}(t,Y_{t})\sigma_{i}(t,Y_{t}),\sigma_{j}(t,Y_{t})\rangle d\langle B^{i},B^{j}\rangle_{t}\nonumber\\
&+\int_{R_{0}^{d}}e^{\lambda t}[V(t,Y_{t^{-}}+K(t,Y_{t},z))-V(t,Y_{t^{-}})]L(dt,dz).
\end{align}
Thus, we have
\begin{align}
&e^{\lambda t}V(t,Y_{t})=e^{\lambda t_{0}}V(t_{0},Y_{0})+\int_{t_{0}}^{t}e^{\lambda r}[\lambda V(r,Y_{r})+LV_{r}(t,Y_{r})]dr+\int_{t_{0}}^{t}e^{\lambda r}\langle V_{Y}(r,Y_{r}),\sigma_{j}(r,Y_{r})\rangle dB_{r}^{j}\nonumber\\
&+M_{t}^{t_{0}}+P_{t}^{t_{0}}.
\end{align}
Since the last three terms are $G$-martingale,  then take expectation on the two sides, we get
\begin{align}
&\hat{\mathbb{E}}[e^{\lambda t}V(t,Y_{t})]\leq\hat{\mathbb{E}}[e^{\lambda t_{0}}V(t_{0},Y_{0})]+\hat{\mathbb{E}}[\int_{t_{0}}^{t}e^{\lambda r}[\lambda V(r,Y_{r})+LV_{r}(t,Y_{r})]dr].
\end{align}
From condition (d), we have
\begin{align}
&\hat{\mathbb{E}}[e^{\lambda t}V(t,Y_{t})]\leq\hat{\mathbb{E}}[e^{\lambda t_{0}}V(t_{0},Y_{0})].
\end{align}
Since $V(t,Y)\leq C_{4}|Y|^{2}$, it holds that $\hat{\mathbb{E}}[e^{\lambda t_{0}}V(t_{0},Y_{0})]\leq C_{4}\hat{\mathbb{E}}[|Y_{0}|^{2}]$  and
$$\hat{\mathbb{E}}[e^{\lambda t}V(t,Y_{t})]\leq  C_{4}\hat{\mathbb{E}}[|Y_{0}|^{2}]e^{\lambda t_{0}},$$
so
$$\hat{\mathbb{E}}|Y_{t}|^{2}\leq\frac{\hat{\mathbb{E}}[V(t,Y_{t})]}{C_{3}}\leq  \frac{C_{4}}{C_{3}}\hat{\mathbb{E}}[|Y_{0}|^{2}]e^{-\lambda (t-t_{0})}.$$

\begin{theorem}
Assume that there exist a $V\in v_{0}$, constants $C_{4}>C_{3}>0$ and $\lambda> 0$ such that

(c)  $C_{3}|Y|^{2}\leq V(t,Y)\leq C_{4}|Y|^{2}$ for all $t\geq t_{0},Y\in R^{n},$

(d1) $LV(t,Y_{t})\leq (-\lambda+\lambda_{1}(t)) V(t,Y_{t})$, where $\lambda_{1}:[t_{0}，+\infty)\rightarrow R$ is a continuous function such that $\int_{t_{0}}^{+\infty}\lambda_{1}^{+}(s)ds<\infty.$

Then, the trivial solution of system (4.1) is mean square exponentially stable.
\end{theorem}
{\bf Proof~~~}
Since $\int_{t_{0}}^{+\infty}\lambda_{1}^{+}(s)ds<\infty$, it follows that there exists a positive constant $M_{1}$, such that $\int_{t_{0}}^{+\infty}\lambda_{1}^{+}(s)ds<M_{1}.$
With the same discussions as in Theorem 4.1, applying the $G$-It\^{o} formula to $e^{\lambda t}V(t, Y_{t})$, we obtain
\begin{align}
&\hat{\mathbb{E}}[e^{\lambda t}V(t,Y_{t})]\leq\hat{\mathbb{E}}[e^{\lambda t_{0}}V(t_{0},Y_{0})]+\hat{\mathbb{E}}[\int_{t_{0}}^{t}e^{\lambda r}[\lambda V(r,Y_{r})+LV_{r}(t,Y_{r})]dr].
\end{align}
From condition (d1), we have
\begin{align}
&\hat{\mathbb{E}}[e^{\lambda t}V(t,Y_{t})]\leq\hat{\mathbb{E}}[e^{\lambda t_{0}}V(t_{0},Y_{0})]+\hat{\mathbb{E}}[\int_{t_{0}}^{t}e^{\lambda r}\lambda_{1}(r) V(r,Y_{r})dr]\nonumber\\
&\leq \hat{\mathbb{E}}[e^{\lambda t_{0}}V(t_{0},Y_{0})]+\int_{t_{0}}^{t}\lambda_{1}^{+}(r)\hat{\mathbb{E}}[e^{\lambda r} V(r,Y_{r})]dr.
\end{align}
By the Gronwall inequality and condition (c), we have
\begin{align}
&\hat{\mathbb{E}}[e^{\lambda t}V(t,Y_{t})]\leq\hat{\mathbb{E}}[e^{\lambda t_{0}}V(t_{0},Y_{0})]e^{\int_{t_{0}}^{t}\lambda_{1}^{+}(r)dr}\nonumber\\
&\leq C_{4}\hat{\mathbb{E}}[|Y_{0}|^{2}]e^{\lambda t_{0}}e^{\int_{t_{0}}^{t}\lambda_{1}^{+}(r)dr}\leq M_{1} C_{4}\hat{\mathbb{E}}[|Y_{0}|^{2}]e^{\lambda t_{0}},
\end{align}
so
\begin{align}
&\hat{\mathbb{E}}|Y_{t}|^{2}\leq\frac{\hat{\mathbb{E}}[V(t,Y_{t})]}{C_{3}}\leq \frac{M_{1} C_{4}}{C_{3}}\hat{\mathbb{E}}[|Y_{0}|^{2}]e^{-\lambda (t-t_{0})}.
\end{align}

The following theorem shows that the solution of system (4.1) is quasi sure exponentially stable under some additional conditions.

\begin{theorem}
Assume that there exist a $V\in v_{0}$, positive constants $C_{3},C_{4}, \lambda$ and $\alpha$ such that

(c)  $C_{3}|Y|^{2}\leq V(t,Y)\leq C_{4}|Y|^{2}$ for all $t\geq t_{0},Y\in R^{n},$

(d1) $LV(t,Y_{t})\leq (-\lambda+\lambda_{1}(t)) V(t,Y_{t})$, where $\lambda_{1}:[t_{0}，+\infty)\rightarrow R$ is a continuous function such that $\int_{t_{0}}^{+\infty}\lambda_{1}^{+}(s)ds<\infty.$

(e) $\hat{\mathbb{E}}[|b(t,Y_{t})|^{2}+|h_{ij}(t,Y_{t})|^{2}+|\sigma_{i}(t,Y_{t})|^{2}+\sup_{v\in\mathcal{V}}\int_{R_{0}^{d}}|K(t,Y_{t},z)|^{2}v(dz)]<\alpha\hat{\mathbb{E}}[|Y_{t}|^{2}].$

Then, the trivial solution of system (4.1) is quasi sure exponentially stable.
\end{theorem}
{\bf Proof~~~} The conditions of Theorem 4.3 imply that all the conditions of Theorem 4.2 hold, so the solution of system (4.1) is mean square exponentially stable. Therefore, there exist a positive constant $M_{2}$  such that
\begin{align}
&\hat{\mathbb{E}}[|Y_{t}|^{2}]\leq M_{2}e^{-\lambda (t-t_{0})}.
\end{align}
In fact,
\begin{align}
&Y_{t+s}=Y_{t}+\int_{t}^{t+s}b(r,Y_{r})dr+\int_{t}^{t+s}h_{ij}(r,Y_{r})d\langle B^{i}, B^{j}\rangle_{r}\nonumber\\
&+\int_{t}^{t+s}\sigma_{i}(r,Y_{r}) dB^{i}_{r}+\int_{t}^{t+s}\int_{R_{0}^{d}}K(r,Y_{r},z)L(dr,dz).
\end{align}
By $C_{r}$-inequality, it holds that
\begin{align}
&|Y_{t+s}|^{2}\leq5[|Y(t)|^{2}+|\int_{t}^{t+s}b(r,Y_{r})dr|^{2}+|\int_{t}^{t+s}h_{ij}(r,Y_{r})d\langle B^{i}, B^{j}\rangle_{r}|^{2}\nonumber\\
&+|\int_{t}^{t+s}\sigma_{i}(r,Y_{r}) dB^{i}_{r}|^{2}+|\int_{t}^{t+s}\int_{R_{0}^{d}}K(r,Y_{r},z)L(dr,dz)|^{2}].
\end{align}
Furthermore, we obtain
\begin{align}
&\hat{\mathbb{E}}[\sup_{0\leq s\leq \tau}|Y_{t+s}|^{2}]\leq5\{\hat{\mathbb{E}}|Y_{t}|^{2}+\hat{\mathbb{E}}[\int_{t}^{t+\tau}|b(r,Y_{r})|dr]^{2}+\hat{\mathbb{E}}[\sup_{0\leq s\leq \tau}|\int_{t}^{t+s}h_{ij}(r,Y_{r})d\langle B^{i},B^{j}\rangle_{r}|^{2}]\nonumber\\
&+\hat{\mathbb{E}}[\sup_{0\leq s\leq \tau}|\int_{t}^{t+s}\sigma_{i}(r,Y_{r}) dB^{i}_{r}|^{2}]+\hat{\mathbb{E}}[\sup_{0\leq s\leq \tau}|\int_{t}^{t+s}\int_{R_{0}^{d}}K(r,Y_{r},z)L(dr,dz)|^{2}]\},
\end{align}
where $\tau$ is a positive constant.

By the H\"{o}lder inequality, (e) and (4.12), we have
\begin{align}
&\hat{\mathbb{E}}[\int_{t}^{t+\tau}|b(r,Y_{r})|dr]^{2}\leq \tau\int_{t}^{t+\tau}\hat{\mathbb{E}}|b(r,Y_{r})|^{2}dr\nonumber\\
&\leq \tau\int_{t}^{t+\tau}\alpha\hat{\mathbb{E}}|Y_{r}|^{2}dr\leq \frac{\alpha M_{2}\tau}{\lambda}e^{-\lambda (t-t_{0})}.
\end{align}
By Lemma 2.5, (e) and (4.12), we have
\begin{align}
&\hat{\mathbb{E}}[\sup_{0\leq s\leq \tau}|\int_{t}^{t+s}h_{ij}(r,Y_{r}))d\langle B^{i}, B^{j}\rangle_{r}|^{2}]\leq C'_{2}\tau\int_{t}^{t+\tau}\hat{\mathbb{E}}|h_{ij}(r,Y_{r})|^{2}dr\nonumber\\
&\leq C'_{2}\alpha\tau\int_{t}^{t+\tau}\hat{\mathbb{E}}|Y_{r}|^{2}dr\leq\frac{ C'_{2}\alpha\tau M_{2}}{\lambda}e^{-\lambda (t-t_{0})}.
\end{align}
By Lemma 2.4, (e) and (4.12), we have
\begin{align}
&\hat{\mathbb{E}}[\sup_{0\leq s\leq \tau}|\int_{t}^{t+s}\sigma_{i}(r,Y_{r})dB^{i}_{r}|^{2}]\leq C_{2}\int_{t}^{t+\tau}\hat{\mathbb{E}}|\sigma_{i}(r,Y_{r})|^{2}dr\nonumber\\
&\leq C_{2}\alpha\int_{t}^{t+\tau}\hat{\mathbb{E}}|Y_{r}|^{2}dr\leq\frac{ C_{2}\alpha M_{2}}{\lambda}e^{-\lambda (t-t_{0})}.
\end{align}
Similarly, by Lemma 2.3, we have
\begin{align}
&\hat{\mathbb{E}}[\sup_{0\leq s\leq \tau}|\int_{t}^{t+s}\int_{R_{0}^{d}}K(r,Y_{r},z)L(dr,dz)|^{2}]\}\leq C_{\tau}\int_{t}^{t+\tau}\hat{\mathbb{E}}[\sup_{v\in \mathcal{V}}\int_{R_{0}^{d}}|K(r,Y_{r},z)|^{2}v(dz)]dr\nonumber\\
&\leq C_{\tau}\alpha\int_{t}^{t+\tau}\hat{\mathbb{E}}|Y_{r}|^{2}dr\leq\frac{ C_{\tau}\alpha M_{2}}{\lambda}e^{-\lambda (t-t_{0})}.
\end{align}
Substituting (4.16)-(4.19) into (4.15), we get
\begin{align}
&\hat{\mathbb{E}}[\sup_{0\leq s\leq \tau}|Y_{t+s}|^{2}]\leq M_{3}e^{-\lambda t},
\end{align}
where $M_{3}>0$ is a constant. Then for $n=1,2,\ldots,$ it follows that
\begin{align}
&\hat{\mathbb{E}}[\sup_{n\tau\leq t\leq (n+1)\tau}|Y_{t}|^{2}]\leq M_{3}e^{-\lambda n\tau}.
\end{align}
Hence, for an arbitrary $\varepsilon\in (0,\lambda)$ and $n\in \mathbb{N}$, from Lemma 2.1, we derive that
$$c(w:\sup_{n\tau\leq t\leq (n+1)\tau}|Y_{t}|^{2}>e^{-(\lambda-\varepsilon)n\tau})\leq M_{3}e^{-\varepsilon n\tau}.$$
Using the Borel-Cantelli Lemma,  we deduce that there exists a $n_{0}(w)$ such that for almost all $w\in \Omega$, $n>n_{0}(w),$
$$\sup_{n\tau\leq t\leq (n+1)\tau}|Y_{t}|^{2}\leq e^{-(\lambda-\varepsilon)n\tau},  \quad q.s.$$
Then, for $n\tau\leq t\leq (n+1)\tau$,
\begin{align}
&\frac{\ln|Y_{t}|}{t}=\frac{\ln|Y_{t}|^{2}}{2t}\leq \frac{\ln\sup_{n\tau\leq t\leq (n+1)\tau}|Y_{t}|^{2}}{2n\tau}\leq \frac{-(\lambda-\varepsilon)}{2}  \quad q.s.
\end{align}
Taking $lim sup$ in (4.22) leads to quasi-surely exponential estimate, that is,
\begin{align}
&\limsup_{t\rightarrow\infty}\frac{\ln|Y_{t}|}{t}\leq \frac{-(\lambda-\varepsilon)}{2}  \quad q.s.
\end{align}
Letting $\varepsilon\rightarrow 0$, we obtain the desired result.

\section{An example }  \label{s4}
In this section, an example is given to illustrate the effectiveness of the obtained results  in section 4.

We consider the following one dimension $G$-stochastic differential equations
\begin{equation}
\begin{cases}
dY_{t}=-2Y_{t}dt-\frac{\sin^{2}t}{2(1+t^{2})}Y_{t}d\langle B, B\rangle_{t}+(1+\frac{|sint|}{\sqrt{1+t^{2}}}) Y_{t} dB_{t}+\int_{R_{0}^{d}}Y_{t}R(z)L(dt,dz),
\\
Y_{t_{0}}=Y_{0}
\end{cases}
\end{equation}
where $B$ is a one dimension $G$-Brownian motion. Let $b(t,Y_{t})=-2Y_{t},$  $h(t,Y(t))=-\frac{\sin^{2}t}{2(1+t^{2})}Y_{t}$,  $\sigma(t,Y_{t})=(1+\frac{|sint|}{\sqrt{1+t^{2}}})Y_{t}$, $K(t,Y(t),z)=R(z)Y_{t}$, $V(t,Y)=|Y|^{2}$
and the function $R(z)$ is assumed to satisfy
 $$\sup_{v\in \mathcal{V}}\int_{R_{0}^{d}}|R(z)|^{2}v(dz)< k  $$
and $$\sup_{v\in \mathcal{V}}\int_{R_{0}^{d}}[(1+R(z))^{2}-1]v(dz)=:l<3.$$
Then we have
$$V_{Y}(t,Y)b(t,Y_{t})=-4Y^{2}，$$
$$V_{Y}(t,Y)h(t,Y_{t})=-\frac{\sin^{2}t}{1+t^{2}}Y^{2},$$
$$V_{YY}(t,Y)\sigma^{2}(t,Y_{t})=2(1+\frac{|\sin t|}{\sqrt{1+t^{2}}})^{2}Y^{2},$$
\begin{align}
&LV(t,Y)=-4Y^{2}-\frac{\sin^{2}t}{1+t^{2}}Y^{2}+(1+\frac{|\sin t|}{\sqrt{1+t^{2}}})^{2}Y^{2}+\sup_{v\in \mathcal{V}}\int_{R_{0}^{d}}[(1+R(z))^{2}-1]Y^{2}v(dz)\nonumber\\
&\leq -3Y^{2}+\frac{2|\sin t|}{\sqrt{1+t^{2}}}Y^{2}+\sup_{v\in \mathcal{V}}\int_{R_{0}^{d}}[(1+R(z))^{2}-1]v(dz)Y^{2}
\nonumber\\
&\leq (-3+l)Y^{2}+\frac{2|\sin t|}{\sqrt{1+t^{2}}}Y^{2}.
\end{align}
Let $\lambda_{1}(t)=\frac{2|\sin t|}{\sqrt{1+t^{2}}}$, we can see $\int_{t_{0}}^{\infty}\lambda_{1}^{+}(s)ds<\infty$. From Theorem 4.2, we easily know that the solution of system (5.1)  is exponentially stable in mean square. Moreover
$$|b(t,Y(t))|^{2}=4 |Y|^{2}, \quad  \quad |h(t,Y(t))|^{2}\leq |Y|^{2},  \quad  \quad|\sigma(t,Y(t))|^{2}\leq 4|Y|^{2},$$
$$|\sup_{v\in \mathcal{V}}\int_{R_{0}^{d}}K(t,Y,z)v(dz)|^{2}\leq k|Y|^{2}.$$
So, we have
$$\hat{\mathbb{E}}[|b(t,Y_{t})|^{2}+|h(t,Y_{t})|^{2}+|\sigma(t,Y_{t})|^{2}+\sup_{v\in\mathcal{V}}\int_{R_{0}^{d}}|K(t,Y_{t},z)|^{2}v(dz)]<(9+k)\hat{\mathbb{E}}[|Y_{t}|^{2}].$$
From Theorem 4.3,  letting $\alpha=9+k$,  we easily know that the solution of system (5.1) is quasi sure exponentially stable.


\renewcommand{\theequation}{A\arabic{equation}}
\setcounter{equation}{0}
%
%

\end{document}